\documentclass[a4paper,reqno,12pt]{amsart}

\usepackage[utf8]{inputenc}
\usepackage{microtype}

\usepackage[margin=1.2in]{geometry}

\usepackage{amssymb}
\usepackage{mathrsfs}
\usepackage{mathtools}

\usepackage{color}
\usepackage{tikz}
\usetikzlibrary{topaths}

\usepackage{hyperref}

\newtheorem{theorem}{Theorem}[section]
\newtheorem{lemma}[theorem]{Lemma}
\newtheorem{proposition}[theorem]{Proposition}

\newtheorem{observation}[theorem]{Observation}
\newtheorem*{ntheorem}{Theorem}

\theoremstyle{definition}

\newtheorem{remark}[theorem]{Remark}

\newtheorem*{nremark}{Remark}

\newcommand{\Z}{\mathbb{Z}}

\newcommand{\Q}{\mathbb{Q}}
\newcommand{\R}{\mathbb{R}}
\newcommand{\F}{\mathbb{F}}
\newcommand{\I}{^{-1}}
\newcommand{\Frame}{\mathcal{F}}
\newcommand{\Ocal}{\mathcal{O}}
\newcommand{\sector}[1]{\mathscr{#1}}

\DeclareMathOperator{\SL}{SL}
\DeclareMathOperator{\GL}{GL}
\DeclareMathOperator{\lk}{lk}
\DeclareMathOperator{\CAT}{CAT}
\DeclareMathOperator{\diag}{diag}
\DeclareMathOperator{\Min}{Min}
\newcommand{\abs}[1]{\lvert #1 \rvert}
\newcommand{\GroupScheme}{\mathbf{G}}

\newcommand{\defeq}{\mathbin{\vcentcolon =}}

\numberwithin{equation}{section}
\begin{document}

\title{A free subgroup in the image of the 4-strand Burau representation}
\date{\today}

\author[S.~Witzel]{Stefan Witzel}
 \address{Mathematical Institute, University of M\"unster,
   Einsteinstra\ss{}e 62, 48149 M\"unster, Germany}
 \email{s.witzel@uni-muenster.de}

\author[M.~C.~B.~Zaremsky]{Matthew C.~B.~Zaremsky}
 \address{Mathematical Institute, University of M\"unster,
   Einsteinstra\ss{}e 62, 48149 M\"unster, Germany}
\email{mzare\_01@uni-muenster.de}

\begin{abstract}
It is known that the Burau representation of the $4$-strand braid group is faithful if and only if certain matrices $f$ and $k$ generate a (non-abelian) free group. Regarding $f$ and $k$ as isometries of a euclidean building we show that $f^3$ and $k^3$ generate a free group. We give two proofs, one utilizing the metric geometry of the building, and the other using simplicial retractions.
\end{abstract}

\maketitle
\thispagestyle{empty}

It is a longstanding open problem to determine whether the Burau representation of the $4$-strand braid group is faithful. For braid groups $B_n$, this is the only open case; for $n<4$ the Burau representation of $B_n$ is faithful \cite{magnus67}, and for $n>4$ it is not \cite{moody91, long93, bigelow99}. For $n=4$ the faithfulness question is equivalent to the question of whether a certain pair of matrices $f,k\in\SL_3(\Z[t,t\I])$ generate a non-abelian free group \cite[Theorem~3.19]{birman74} (see also \cite[Proposition~4.1]{alperin02}). Concrete interest in the faithfulness of the $4$-strand Burau representation stems from the expectation that a non-trivial element of the kernel would give rise to a non-trivial knot with trivial Jones polynomial (in fact trivial HOMFLY polynomial) \cite[Conjecture~3.2]{bigelow02}. Using the explicit geometry of the relevant euclidean building we show

\begin{ntheorem}
If $m,n\ge 3$ then $f^m$ and $k^n$ generate a free group of rank $2$.
\end{ntheorem}

Alperin has informed us that he has also calculated that some powers of~$f$ and~$k$ generate a free group but did not publish his proof. Alperin, Farb and Noskov \cite{alperin02} give a condition under which a pair of hyperbolic isometries~$f$ and~$k$ of a $\CAT(0)$ space~$X$ generate a copy of $F_2$. The condition includes the requirement that, if axes $A_f$ and $A_k$, of $f$ and $k$ respectively, meet at a single vertex~$v$, then the angles between~$A_f$ and $A_k$ at $v$ must be at least $\pi$. In case $X$ is a euclidean building this condition translates to a statement about opposition in the spherical building that is the link of $v$.

\begin{figure}[ht]
\centering
\includegraphics{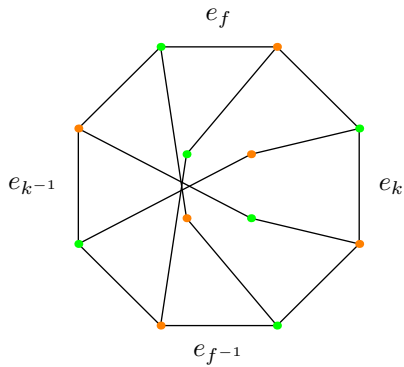}
\caption{Part of the link of $v$. For each of the elements $g \in \{f,k,f\I,k\I\}$ the vertex $v$ is moved by $g$ to a vertex adjacent to the edge labeled $e_g$. Note that each edge spans a regular triangle with $v$, so that the angle between the barycenter of $e_{f^\pm}$ and $e_{k^\pm}$ is $2\pi/3$.}
\label{fig:link_of_v}
\end{figure}

Unfortunately this local opposition requirement is not satisfied by the particular~$f$ and~$k$ of interest for the $4$-strand Burau representation; see Figure~\ref{fig:link_of_v}. The local angles are $2\pi/3$ rather than $\pi$, and this proved to be a crucial impediment to the methods used in \cite{alperin02}. In the present work we show that, while $A_f$ and $A_k$ have angles of $2\pi/3$ locally at $v$, they nonetheless behave like opposites shortly after departing from $v$. It follows that some powers of $f$ and $k$ have to span a free group, and inspection shows that the least powers deducible from this observation are the cubes.

It is worth noting that reducing the $4$-strand Burau representation modulo $2$ or $3$ leads to a representation that is known not to be faithful \cite{cooper97,cooper98}. Therefore, any proof of faithfulness of the Burau representation should use characteristic $0$ or at least large characteristic in an essential way, however:

\begin{nremark}
The proof of the theorem is independent of characteristic $0$. In particular, the theorem also applies to the representations in $\SL_3(\F_p[t,t^{-1}])$ obtained by reducing the Burau representation mod $p$.
\end{nremark}

\medskip

In Section~\ref{sec:buildings} we provide some general background on buildings and their isometries. Facts related to the specific isometries we are interested in are established in Section~\ref{sec:special_case}. The theorem is proved in Section~\ref{sec:pong} using two different techniques, one a metric approach in the spirit of \cite{alperin02} and the other a simplicial approach using building retractions.

\subsection*{Acknowledgments}

We are grateful to Roger Alperin for introducing us to this problem, and for helpful discussions and suggestions. Both authors were supported by the SFB~878 in M\"unster and the second named author also by the SFB~701 in Bielefeld, and this support is gratefully acknowledged.


\section{Buildings and isometries}\label{sec:buildings}

We recall some definitions about $\CAT(0)$ spaces from \cite{bridson99} and about buildings from \cite{abramenko08}.
The $\CAT(0)$ inequality expresses that geodesic triangles in a metric space are at most as thick as their comparison triangles in euclidean space. It gives a way to characterize $\CAT(0)$ spaces as those geodesic metric spaces that have non-positive curvature; for a definition see \cite[Chapter~II.1]{bridson99}. One feature of $\CAT(0)$ spaces is that they have a well defined visual boundary $X^\infty$ \cite[Chapter~II.8]{bridson99}.

Prominent examples of $\CAT(0)$ spaces include symmetric spaces of non-compact type and euclidean buildings. A \emph{euclidean building} is a simplicial complex $X$ equipped with a $\CAT(0)$ metric in which any two points share a top-dimensional flat, or \emph{apartment}, such that any two such apartments are isometric via an isometry fixing their intersection. Every apartment is a euclidean Coxeter complex of the same type. The maximal simplices are called \emph{chambers} and the simplices of codimension one are called \emph{panels}. See \cite[Chapter~11]{abramenko08} for more definitions and details. The boundary~$X^\infty$ of a euclidean building is a spherical building, in particular, it is itself a simplicial complex. Maximal simplices are called \emph{chambers at infinity}.

\medskip

\textbf{Hyperbolic isometries.} We will be interested in hyperbolic isometries of euclidean buildings. We recall the relevant definitions from \cite[Chapter~II.6]{bridson99}. Attached to an isometry $g$ of a $\CAT(0)$ space $X$ are the following data. The \emph{displacement function} $d_g \colon X \to \R$ is given by $d_g(x) = d(x,g.x)$. This gives rise to the \emph{translation length} $\abs{g} = \inf \{d_g(x) \mid x\in X\}$ and the \emph{min set} $\Min(g) = \{x \in X \mid d_g(x) = \abs{g}\}$. Now $g$ is called \emph{hyperbolic} if $\abs{g} > 0$ and $\Min(g) \ne \emptyset$, that is, if the translation length is positive and is attained at some point. Note that the metric on $X$ is convex and therefore $d_g$ is a convex function. In particular, $\Min(g)$, being a sublevel set of a convex function, is convex. The min set of a hyperbolic isometry has a very special structure.

\begin{proposition}\cite[Theorem~6.8]{bridson99}\label{prop:axes}
If $g$ is a hyperbolic isometry then $\Min(g)$ is isometric to a product $Y \times \R$, and the action is given by $g.(y,t) = (y,t+\abs{g})$.
\end{proposition}

Each set of the form $\{y\} \times \R$ in the proposition is called an \emph{axis} of $g$. The boundary of each axis consists of the \emph{attracting limit point} $\xi^+_g \defeq \lim\limits_{n \to \infty} g^n.x \in X^\infty$ and the \emph{repelling limit point} $\xi^-_g \defeq \lim\limits_{n \to -\infty} g^n.x \in X^\infty$ of $g$, whose definitions do not depend on the point $x$.

The structure theory alone is enough to make the following observation.

\begin{observation}[Min sets are apartments]\label{obs:min_sets_apts}
 Let $g$ be a type-preserving hyperbolic isometry of a euclidean building $X$ and assume that the limit points $\xi^\pm_g$ lie in the interior of chambers at infinity $C^\pm$. Then $\Min(g)$ is an apartment $\Sigma$, more precisely,~$\Sigma$ is the unique apartment that contains $C^\pm$ in its boundary.
\end{observation}

\begin{proof}
 First note that $C^\pm$ are opposite chambers because $\xi^\pm_g$ are opposite ends of an axis. Since $g$ acts by metric and cellular automorphisms, $\Min(g)$ has to be a convex subcomplex of $X$ that contains the chambers $C^\pm$ in its boundary. Therefore $\Sigma \subseteq \Min(g)$ and $g$ has to act as a translation on $\Sigma$. On the other hand, direct inspection shows that the interior points of chambers meeting $\Sigma$ but not contained in it are moved farther than $\abs{g}$, and by convexity of $\Min(g)$ this is enough to conclude that $\Min(g) \subseteq \Sigma$.
\end{proof}

The importance of euclidean buildings stems from the fact that they arise naturally from groups over valued fields. More precisely if $k$ is a field with a non-archimedean valuation $\nu$ and $\GroupScheme$ is a reductive $k$-group then there is an associated Bruhat--Tits building $X$ on which $\GroupScheme(k)$ acts by cellular isometries \cite{bruhat72,bruhat84}. We will only need the case of $\GroupScheme = \SL_n$ where an explicit description in terms of lattices is available; see \cite[Chapter~9]{ronan89} or \cite[Section~6.9]{abramenko08}. In our case $k = \Q(t)$ with~$\nu$ the valuation at infinity, so that the valuation ring is $\Q[t\I]$ and a uniformizing element is $\pi = t^{-1}$.

\medskip

\textbf{Galleries, Weyl distance, projections.} A sequence of successively adjacent chambers is called a \emph{gallery}. Recording the types of panels crossed by a gallery defines an element of the associated Coxeter group. In this way, to any two chambers $C$ and $D$ one can assign a \emph{Weyl distance} $\delta(C,D)$, which is the element associated to a minimal gallery connecting $C$ to $D$.

A key property of buildings is that given a simplex $A$ and a chamber $C$, there is a unique chamber $D \ge A$ such that every minimal gallery from $C$ to a chamber containing $A$ has to pass through $D$. This chamber $D$ is called \emph{(building theoretic) projection} of $C$ onto $A$.

\medskip

\textbf{Retractions.} The axiom that in a building any two chambers share an apartment gives rise to the important notion of the \emph{retraction} of a building $X$ onto an apartment~$\Sigma$, based at a chamber $C \subseteq \Sigma$. It is a continuous map $\rho_{\Sigma,C} \colon X \to \Sigma$, characterized by the property that it fixes $C$ pointwise and is a simplicial isomorphism onto $\Sigma$ when restricted to any apartment containing $C$ \cite[Definition~4.38]{abramenko08}.

The retraction $\rho_{\Sigma,C}$ preserves distances to points in $C$, and in general is distance non-increasing. Combinatorially speaking, the image under $\rho_{\Sigma,C}$ of a minimal gallery~$\Gamma$ to~$C$ is a gallery of the same type as $\Gamma$. In particular the Weyl distance $\delta(D,C)$ from a chamber $D$ to $C$ equals $\delta(\rho_{\Sigma,C}(D),C)$.

\medskip

\textbf{Roots, walls, sectors.} A \emph{root} in an apartment is a half-space that is also a subcomplex. The boundary $\partial \alpha$ of a root is a \emph{wall}. A \emph{sector} is an intersection of roots whose visual boundary is a chamber at infinity. It is a simplicial polyhedron and we call the unique vertex of that polyhedron the \emph{tip}. There is a unique chamber in $\sector{S}$ containing the tip.


\section{The isometries of interest for the Burau representation}\label{sec:special_case}

Consider the euclidean building $X$ associated to $\SL_3(\Q(t))$ with respect to the valuation $\nu$ at infinity. The two isometries that we are interested in are given by the matrices
$$f \defeq \begin{pmatrix}
            t & 0 & 0 \\
            0 & 1 & 0 \\
            0 & 0 & t\I
           \end{pmatrix} \text{ and }
  k \defeq \begin{pmatrix}
            0 & -1 - t & -t\I - 1 - t \\
            0 & t\I + 1 + t & t^{-2} + t\I + 1 + t \\
            1 & 0 & 0
           \end{pmatrix}\text{.}
$$
These matrices are conjugate, for example $k=sfs\I$ where
$$s \defeq \begin{pmatrix}
            1 & 1 & t\I \\
            -(t^{-2} + 1) & -(t\I + 1) & -(t^{-2} + 1) \\
            t\I & 1 & 1
           \end{pmatrix}\text{.}$$
Our $f$ and $k$ are simultaneous conjugates of the matrices denoted $f$ and $k$ in \cite{alperin02}, namely ours are obtained from those via conjugating by $\left(\begin{smallmatrix} 0 & 0 & 1 \\ t\I & 0 & 0 \\ 0 & 1 & 0\end{smallmatrix}\right)$ from the left. (Our~$s$ is unrelated to the matrix denoted $s$ in \cite{alperin02}, but that is insignificant.) Those in turn were obtained from two matrices described in \cite[Theorem~3.19]{birman74} (note the relevant erratum in \cite{birman75}). For completeness, we have calculated that to obtain our $f$ and $k$ from the (post-erratum) matrices in \cite{birman74}, call them $a$ and $b$, one replaces $t$ by $-t$ and then conjugates from the left by $\left(\begin{smallmatrix} 0 & 1 & t\I - 1 \\ 0 & -t^{-2} - 1 & 0 \\ t-1 & 1 & 0\end{smallmatrix}\right)$, after which $a$ becomes $f\I$ and $b$ becomes $k$.

\medskip

In particular, the $4$-strand Burau representation is faithful if and only if $\langle f,k\rangle \cong F_2$.

\subsection{Apartments for the isometries}\label{sec:apts_for_isoms}

In this section we describe some apartments that are distinguished by $f$ and $k$.

The min set for $f$ (respectively $k$) is a euclidean apartment $\Sigma_f$ (respectively $\Sigma_k$) by Observation~\ref{obs:min_sets_apts}. Since $f$ is diagonal, $\Sigma_f$ is actually the standard apartment defined by the frame $\Frame_f \defeq \{[e_1],[e_2],[e_3]\}$. Then since $k=sfs\I$, we have $\Sigma_k=s\Sigma_f$, so $\Sigma_k$ is defined by the frame $\Frame_k \defeq \{[se_1],[se_2],[se_3]\}$.

In \cite[Section~4.1]{alperin02} the two apartments of interest are shown to intersect at precisely one vertex. The proof given there has a minor sign error, and since we are dealing with different matrices anyway, we will sketch a proof here. Recall that $\pi=t\I$ is a uniformizing element.

\begin{lemma}\label{lem:apts_meet}
 $\Sigma_f\cap\Sigma_k$ consists only of the lattice class $v\defeq [[\Ocal e_1 + \Ocal e_2 + \Ocal e_3]]$.
\end{lemma}

\begin{proof}
 We are looking for integers $a_i,b_i$ for $i=1,2,3$ such that
 $$s\begin{pmatrix}
     t^{-b_1} & 0 & 0 \\
     0 & t^{-b_2} & 0 \\
     0 & 0 & t^{-b_3}
    \end{pmatrix} = 
    \begin{pmatrix}
     t^{a_1} & 0 & 0 \\
     0 & t^{a_2} & 0 \\
     0 & 0 & t^{a_3}
    \end{pmatrix}m
$$
for some $m\in\GL_3(\Ocal)$. Then we will have $s[[\pi^{b_1}\Ocal e_1 + \pi^{b_2}\Ocal e_2 + \pi^{b_3}\Ocal e_3]]=[[\pi^{-a_1}\Ocal e_1 + \pi^{-a_2}\Ocal e_2 + \pi^{-a_3}\Ocal e_3]] \in \Sigma_f \cap \Sigma_k$, and this is a precise characterization of the intersection. Since $\nu(\det(s))=0$ and $\nu(\det(m))=0$, we obtain $\sum a_i + \sum b_i = 0$. Also, if $s_{ij}$ is the $i,j$ entry of $s$ we obtain inequalities $a_i+\nu(s_{ij})+b_j \ge 0$ for all~$i,j$. Now we get $0=(a_1+b_1)+(a_2+b_2)+(a_3+b_3)\ge 0$ by looking at the diagonal of $s$, implying that $(a_1+b_1)=0$, $(a_2+b_2)=0$ and $(a_3+b_3)=0$. A similar trick shows that $(a_1+b_2)=0$ and $(a_2+b_1)=0$, and another instance of the trick shows that $(a_2+b_3)=0$ and $(a_3+b_2)=0$. This tells us that any solution must satisfy $a_1=a_2=a_3=-b_1=-b_2=-b_3$. We conclude that $s.v=v$ is the only vertex in~$\Sigma_f\cap\Sigma_k$.
\end{proof}

Axes for $f$ point to opposite points $\xi_f^\pm$ in the apartment at infinity $\Sigma_f^\infty$. A similar statement holds for $k$. Let $A_f$ be the axis for $f$ containing $v$, and similarly define~$A_k$ containing $v$. As seen in \cite{alperin02}, the germs of these axes at $v$ are at an angle of $2\pi/3$ in $\lk(v)$, and so the Strong Schottky Lemma of \cite{alperin02} does not apply. Asymptotically, however, we have the following fact.

\begin{lemma}\label{lem:ends_opp}
 The ends $\xi_f^\pm$ and $\xi_k^\pm$ are all pairwise opposite in the building at infinity~$X^\infty$.
\end{lemma}

\begin{proof}
 We can calculate $\xi_f^\pm$ by looking at the limits of $f^n.v$ and $f^{-n}.v$ as $n\to\infty$. Clearly $f^n.v=[[\pi^{-n}\Ocal e_1 + \Ocal e_2 + \pi^n \Ocal e_3]]$. For $n\ge 0$ these lattice classes are of the form $[[\pi^a\Ocal e_1+\pi^b\Ocal e_2 + \pi^c \Ocal e_3]]$ with $a\le b\le c$, and for $n\le 0$ they are of that form with $a\ge b\ge c$. These rules define sectors in $X$, with tip $v$; the former has as its chamber at infinity the fundamental chamber $[e_1]<[e_1,e_2]$, and the latter has $[e_3]<[e_2,e_3]$. This is explained, e.g., in the proof of Proposition~11.105 in \cite{abramenko08}.
 
 In particular $\xi_f^+$ is the barycenter of the chamber $[e_1]<[e_1,e_2]$ at infinity, and $\xi_f^-$ is the barycenter of $[e_3]<[e_2,e_3]$. Of course then $\xi_k^\pm$ are barycenters of chambers as well, namely $[se_1]<[se_1,se_2]$ and $[se_3]<[se_2,se_3]$. Inspecting the columns of $s$, we see that $se_1$ is not contained in $[e_1,e_2]$, and also that $e_1$ is not contained in $[se_1,se_2]$, so the chambers $[e_1]<[e_1,e_2]$ and $[se_1]<[se_1,se_2]$ are opposite. Similar arguments establish opposition for the other pairs of chambers.
\end{proof}

Since $\xi_f^+$ and $\xi_k^+$ are barycenters of opposite chambers, they are contained in a unique spherical apartment $\Sigma_{f,k}^\infty$ bounding a euclidean apartment $\Sigma_{f,k}$. Knowing a pair of opposite chambers, it is easy to calculate a frame $\Frame_{f,k}$ for $\Sigma_{f,k}^\infty$ -- two of the lines are $[e_1]$ and $[se_1]$, and the third line is $[e_1,e_2]\cap[se_1,se_2]$. The other three frames can be calculated similarly. We collect the calculations in Table~\ref{table:frames}.

\begin{table}[h!]
\centering
\begin{tabular}{r|lll}
 Frame & ~ & Defining Lines \\ \hline
 $\Frame_{f,k}$ & $[e_1]$ & $[(t\I-1)e_1+(1-t\I)e_2]$ & $[te_1-(t\I + t)e_2 + e_3]$ \\ \hline
 $\Frame_{f,k\I}$ & $[e_1]$ & $[(1-t)e_1+(1-t\I)e_2]$ & $[t\I e_1-(t^{-2}+1)e_2+e_3]$ \\ \hline
 $\Frame_{f\I,k}$ & $[e_3]$ & $[(1-t\I)e_2+(1-t)e_3]$ & $[e_1-(t^{-2}+1)e_2+t\I e_3]$ \\ \hline
 $\Frame_{f\I,k\I}$ & $[e_3]$ & $[(1-t\I)e_2+(t\I-1)e_3]$ & $[e_1-(t\I+t)e_2+te_3]$
\end{tabular}\vspace{0.05in}
\caption{Frames for the four apartments $\Sigma_{f^{\pm1},k^{\pm1}}$.}
\label{table:frames}
\end{table}

Note that in some cases our chosen representative vector for the lines in the frame is not the one immediately obtained from $s$; we have attempted to make the frames look as simple as possible.

We can also read off transformations $b_{f^{\pm1},k^{\pm1}}$ taking the standard frame $\Frame_f$ to $\Frame_{f^{\pm1},k^{\pm1}}$, namely
$$b_{f,k}=\begin{pmatrix}
           1 & t\I-1 & t \\
           0 & 1-t\I & -(t\I + t) \\
           0 & 0 & 1
          \end{pmatrix}\text{, }
  b_{f,k\I}=\begin{pmatrix}
           1 & 1-t & t\I \\
           0 & 1-t\I & -(t^{-2}+1) \\
           0 & 0 & 1
          \end{pmatrix}\text{, }$$

$$b_{f\I,k}=\begin{pmatrix}
           1 & 0 & 0 \\
           -(t^{-2} + 1) & 1-t\I & 0 \\
           t\I & 1-t & 1
          \end{pmatrix}\text{, and }
  b_{f\I,k\I}=\begin{pmatrix}
           1 & 0 & 0 \\
           -(t\I + t) & 1-t\I & 0 \\
           t & t\I-1 & 1
          \end{pmatrix}\text{.}
$$

\subsection{Apartments for pairs of isometries}\label{sec:apts_for_pairs}

\begin{figure}[ht]
\centering
\includegraphics{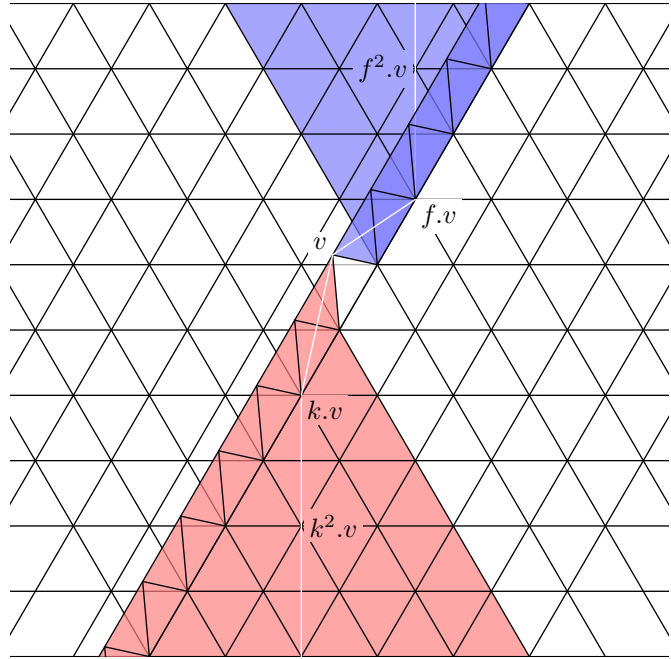}
\caption{The apartment $\Sigma_{f,k}$ plus an extra strip of chambers. The blue region indicates the part of the picture that also lies in $\Sigma_f$, and the red region indicates intersection of the picture with $\Sigma_k$. The white line in the blue part is $[v,\xi^+_f)$, and the white line in the red part is~$[v,\xi^+_k)$.}
\label{fig:sigma_fk}
\end{figure}

In this section we show that $v$ is not in any of the $\Sigma_{f^{\pm1},k^{\pm1}}$, but that for any $n>0$, we have that $f^n.v$ already lies in~$\Sigma_{f,k^{\pm1}}$, with similar statements for $f\I$, $k$ and $k\I$. By inspecting which lattice classes in, say, $\Sigma_{f,k}$ correspond to the points $f^n.v$ and $k^n.v$, we also obtain a helpful precise picture; see Figure~\ref{fig:sigma_fk}.

We know that $f^n.v=[[\pi^{-n}\Ocal e_1 + \Ocal e_2 + \pi^n\Ocal e_3]]$ for any $n$. This lies in $\Sigma_{f,k}$ if and only if there exist integers $a,b,c$ and a matrix $m\in\GL_3(\Ocal)$ such that
$$f^n=b_{f,k}\diag(\pi^a,\pi^b,\pi^c)m\text{.}$$

\begin{proposition}\label{prop:push_into_apts}
 Let $g=f^{\pm1}$ and $h=k^{\pm1}$ and let $n\ge 0$. Then $g^nv\in \Sigma_{g,h}$ if and only if $n>0$, and similarly $h^nv\in \Sigma_{g,h}$ if and only if $n>0$.
\end{proposition}

\begin{proof}
 We will prove that $f^n.v\in \Sigma_{f,k}$ if and only if $n>0$, and the other seven cases work by parallel (though tedious) arguments. Suppose $f^n.v\in \Sigma_{f,k}$, say $f^n=b_{f,k}\diag(\pi^a,\pi^b,\pi^c)m$. Taking valuations of determinants, we get $a+b+c=0$. We also obtain constraining inequalities from the fact that the entries of $m$ must all have non-negative valuation, and hence too the entries of $f^{-n}b_{f,k}\diag(\pi^a,\pi^b,\pi^c)$. We can arrange these inequalities in a matrix, namely
 $$\begin{pmatrix}
    n+a \ge 0 & n+b \ge 0 & n+c-1 \ge 0 \\
    - & b\ge 0 & c-1\ge 0 \\
    - & - & -n+c \ge 0
   \end{pmatrix}\text{.}
$$
From the determinant equality and the inequalities in the second row, we obtain $0=a+b+c \ge a+1$, and from the top-left inequality we then get $0\ge 1-n$, so $n\ge 1$. For the converse, note that if $n\ge 1$ then we can take $a=-n$, $b=0$, $c=n$.
\end{proof}

For $g\in\{f,k\}$ and $\varepsilon\in\{+,-\}$ let $v_g^\varepsilon$ be the vertex $g^{\varepsilon 1}.v$. We find that the ray $[v_f^+,\xi_f^+)$ is precisely the intersection of the axis $A_f$ with either apartment $\Sigma_{f,k^{\pm1}}$, with similar statements for the other rays and apartments, e.g., $[v_k^-,\xi_k^-)$ is the intersection of $A_k$ with $\Sigma_{f^{\pm1},k\I}$.


\section{Ping-Pong}\label{sec:pong}

We retain the definitions and setup from the previous section. In this section we prove the theorem using two different methods. In both cases, the last step is applying the well known Ping-Pong Lemma (cf.~\cite[Lemma~3.1]{alperin02}).

\begin{lemma}[Ping-Pong Lemma]\label{lem:pplemma}
 Let $\Gamma$ be a group acting on a set $X$. Let $g_1,g_2$ be elements of order at least~$3$. Suppose there are disjoint subsets $X_1,X_2$ of $X$ such that, for all $n\neq 0$ and for $i\neq j$, we have $g_i^n(X_j)\subseteq X_i$. Then $g_1$ and $g_2$ generate a copy of~$F_2$.
\end{lemma}

\subsection{Metric approach}\label{sec:metric_approach}

First we will play ping-pong using metric projections. We consider each euclidean apartment $\Sigma$ of type $\widetilde{A}_2$ as a euclidean vector space by randomly choosing an origin so that we have a notion of vectors. The inner product will be denoted $\langle \cdot,\cdot \rangle$. Every root $\alpha$ of $\Sigma$ naturally gives rise to a unit vector $\hat{\alpha}$ that is perpendicular to $\partial \alpha$ and (at each point of $\partial \alpha$) points into $\alpha$. The vector $\hat{\alpha}$ can be thought of as a root in the spherical root system of $\Sigma^\infty$.

We start by proving a lemma that will be useful to verify disjointness of the ping-pong sets. Consider the following configuration. Let $X$ be a building of type~$\widetilde{A}_2$ and let $\Sigma$ be an apartment. Let $\sector{S}$ be a sector with tip $x$. Let $C$ be the chamber at the tip of the opposite sector and let $C'$ be the chamber of the opposite sector adjacent to $C$; see Figure~\ref{fig:bouncing}. Denote by $\rho$ the retraction onto $\Sigma$ based at $C$ or at $C'$. Consider a geodesic path $\gamma \colon [0,\ell] \to X$. We want to put some restrictions on how the path $\bar{\gamma} \defeq \rho \circ \gamma$ can travel. It is known \cite[Lemmas~4.3,~4.4]{kapovich08} that $\bar{\gamma}$ is piecewise geodesic. For it to fail to be a local geodesic at a time $t$ the following conditions must be met:
\begin{enumerate}
\item $\bar{\gamma}(t) \in \partial \alpha$ for some root $\alpha$ that does not contain $C$,
\item the incoming tangent vector to $\bar{\gamma}(t)$ includes an obtuse angle with $\hat{\alpha}$.
\end{enumerate}
In this situation $\bar{\gamma}$ can bounce off of $\partial \alpha$ and continue traveling in $\alpha$. If that happens, the outgoing tangent vector of $\bar{\gamma}$ at $t$ is the incoming one reflected at $\partial \alpha$, in particular, it now includes an acute angle with $\hat{\alpha}$. If $\bar{\gamma}(t)$ lies in several walls (that is, in a face of codimension $>1$) then several of these bouncings can occur at the same time.

\begin{figure}[ht]
\centering
\includegraphics{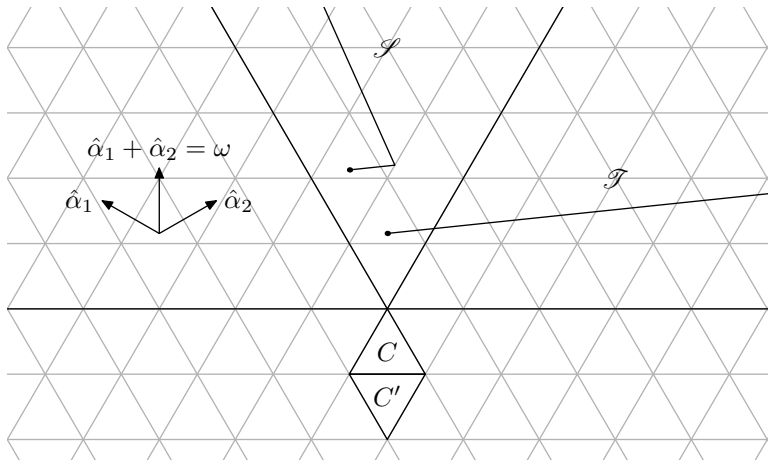}
\caption{Possibilities for $\bar{\gamma}$ in the proof of Lemma~\ref{lem:retractions_of_geodesics}.}
\label{fig:bouncing}
\end{figure}

\begin{lemma}\label{lem:retractions_of_geodesics}
 Let $\sector{S}$ and $\bar{\gamma}$ be as above and let $\omega$ be a vector pointing to the barycenter of $\sector{S}^\infty$. Assume that the starting point of $\bar{\gamma}$ lies in $\sector{S}$ and that the initial tangent vector includes a non-obtuse angle with $\omega$. Then every tangent vector of $\bar{\gamma}$ includes a non-obtuse angle with $\omega$.
\end{lemma}

\begin{proof}
 We prove the case where $\rho$ is based at $C$, the case where $\rho$ is based at $C'$ being a consequence.

 Let $\alpha_1$ and $\alpha_2$ be such that $\sector{S} = \alpha_1 \cap \alpha_2$ and note that up to scaling $\omega$ equals $\hat{\alpha}_1 + \hat{\alpha}_2$. By assumption the initial tangent vector of $\bar{\gamma}(t)$ includes a non-obtuse angle with $\omega$. By the above discussion, as long as $\bar{\gamma}(t) \in \sector{S}$ the only directions in which $\bar{\gamma}$ can bounce off are toward $\hat{\alpha}_1$, $\hat{\alpha}_2$ and $\hat{\alpha}_1+\hat{\alpha}_2$. Since these three roots include an acute angle with $\omega$, this keeps the angle between the tangent vector and $\omega$ non-obtuse (actually any bouncing would make it acute).

 Now suppose that at time $t_0$ the path $\bar{\gamma}$ leaves $\sector{S}$ and enters an adjacent sector~$\sector{T}$ with the same tip. Without loss of generality, $\sector{T}$ is the intersection of $\alpha_1+\alpha_2$ and~$-\alpha_1$. Now as long as $\bar{\gamma}$ stays in $\sector{T}$ it can only bounce off in the direction of $\hat{\alpha}_2$, $-\hat{\alpha}_1$, or $\hat{\alpha}_1 + \hat{\alpha}_2$. We claim that for $t \ge t_0$ the following invariant is preserved

 \smallskip

 \begin{quote}
  Each tangent vector to $\bar{\gamma}(t)$ includes a non-obtuse angle with $\hat{\alpha}_1+\hat{\alpha}_2$ as well as with $-\hat{\alpha}_1$.
 \end{quote}

 \smallskip

 Clearly if this holds, then $\bar{\gamma}(t) \in \sector{T}$ for $t \ge t_0$. To see that the invariant is preserved, note that it holds at time $t_0$ by the assumption that $\bar{\gamma}$ enters $\sector{T}$ at that time. Also, since it holds there, $\bar{\gamma}$ bounces off neither in the direction of $-\hat{\alpha}_1$ nor in the direction of $\hat{\alpha}_1+\hat{\alpha}_2$. Finally, bouncing off in the direction $\hat{\alpha}_2$ does not affect the invariant (actually, it cannot happen either).
\end{proof}

For $g\in \{f,k\}$ let $p_g \colon X \to A_g$ be the metric projection onto $A_g$, and define the set
\[
X_g \defeq p_g\I((v_g^+,\xi_g^+)\cup (v_g^-,\xi_g^-))\text{.}
\]
We want to play ping-pong on $X_f$ and $X_k$. We start by showing disjointness.

\begin{figure}[ht]
\centering
\includegraphics{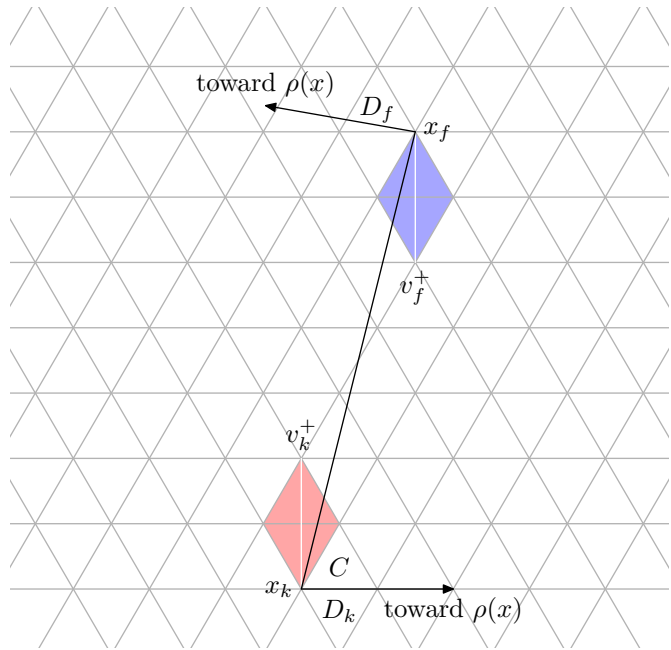}
\caption{The contradiction in the proof of Proposition~\ref{prop:metric_disjointness}: the picture shows part of the image of the retraction of the triangle $[x,x_f] \cup [x_f,x_k] \cup [x,x_k]$ onto $\Sigma$.}
\label{fig:metric_disjointness}
\end{figure} \pagebreak[3]

\begin{proposition}[Metric approach -- disjointness]\label{prop:metric_disjointness}
 The sets $X_f$ and $X_k$ are disjoint.
\end{proposition}

\begin{proof}
Suppose that $x \in X$ is such that $p_f(x) \in (v_f^+,\xi_f^+)$ and that $p_k(x) \in (v_k^+,\xi_k^+)$. There are three other cases which are dealt with in exactly the same way. We write $x_f \defeq p_f(x)$ and $x_k \defeq p_k(x)$ and consider the geodesic triangle spanned by $x$, $x_f$ and $x_k$. Let $D_f$ be a chamber that contains $x_f$ as well as the germ of the geodesic toward $x$. Let $D_k$ be defined similarly. There is an apartment containing~$D_f$ and~$D_k$ which we call $\Sigma$. Being convex, $\Sigma$ also contains the geodesic $[x_f,x_k]$. Moreover, an inspection of $\Sigma_{f,k}$ tells us that the angle at $x_f$ between the direction toward $x_k$ and the direction toward $v_f^+$ is strictly less than $\pi/6$. This means that they point into the same chamber and in particular, the germ of $[x_f,v_f^+]$ in $x_f$ can also be seen in~$\Sigma$. The same reasoning applies to the germ of $[x_k,v_k^+]$ in $x_k$.

Now let $\omega$ be a vector at $x_f$ pointing away from $v_f^+$ (the same direction as from $x_k$ to $v_k^+$). Let $e$ be the edge of $D_k$ that is perpendicular to $\omega$ and let $C$ be the chamber of $\Sigma$ containing $e$ into which $\omega$ points. So $C$ is either $D_k$ or adjacent to it.  Let $\sector{S}$ be the sector in $\Sigma$ into which $\omega$ points and that meets $C$ precisely at its tip. In particular, we can apply Lemma~\ref{lem:retractions_of_geodesics} with $D_k$ being $C$ or $C'$. So let $\rho$ be the retraction onto $\Sigma$ based at $D_k$. We want to retract the triangle in question using~$\rho$, see Figure~\ref{fig:metric_disjointness}. First note that since $x_k$ is the projection of $x$, the angle $\angle_{x_k}(v_k^+,x)$ is non-acute \cite[Proposition~2.4~(4)]{bridson99}, and that $\rho([x_k,x])$ is a geodesic because $x_k \in D_k$, which shows that $\langle x_k-\rho(x),\omega \rangle \ge 0$. Second we can apply Lemma~\ref{lem:retractions_of_geodesics} to the geodesic $[x_f,x]$ 
to get that $\langle  \rho(x)-x_f,\omega \rangle \ge 0$. But as can be seen in $\Sigma_{f,k}$ we also have $\langle x_f-x_k,\omega \rangle >0$, and this implies that $\langle \rho(x) - \rho(x),\omega \rangle)>0$ which is absurd.
\end{proof}

It remains to show containment of the shifted ping-pong sets.

\begin{lemma}[Metric approach -- ping-pong]\label{lem:metric_pong}
 For any $|n|>2$ we have $f^n(X \setminus X_f)\subseteq X_f$ and $k^n(X \setminus X_k)\subseteq X_k$.
\end{lemma}

\begin{proof}
 For $g\in\{f,k\}$ note that the action of $\langle g\rangle$ commutes with $p_g$. Then since $g^{\pm n}([v_g^-,v_g^+])\subseteq (v_g^\pm,\xi_g^\pm)$ for $n>2$, we conclude that $g^n(X\setminus X_g) \subseteq X_g$ for $|n|>2$, and the result follows.
\end{proof}

\begin{proof}[Proof 1 of Theorem]
Let $m,n \ge 3$. We set $g_1 = f^m$, $g_2 = k^n$, $X_1 = X_f$ and $X_2 = X_k$ and want to apply the Ping-Pong Lemma. It is easily seen that $g_1$ and $g_2$ have infinite order. Proposition~\ref{prop:metric_disjointness} shows that $X_1$ and $X_2$ are disjoint and Lemma~\ref{lem:metric_pong} shows that non-trivial powers of $g_i$ move $X_j$ into $X_i$ for $i \ne j$.
\end{proof}

\begin{remark}\label{rmk:metric_limitations}
It seems tempting to replace the open rays $(v_g^\pm,\xi_g^\pm)$ by closed rays $[v_g^\pm,\xi_g^\pm)$ to try to obtain freeness of the group generated by $f^2$ and $k^2$. However, using the analogous definitions, disjointness fails to hold: there exist points $x$ for which $x_f=v_f^+$ and $x_k=v_k^+$. An explicit example of such an $x$ is $x=[[\Ocal b_1 + \pi^{-2}\Ocal b_2 + \pi^{-1}\Ocal b_3]]$ in $\Sigma_{f,k}$, where $b_1,b_2,b_3$ are the three vectors spanning the lines in the frame $\Frame_{f,k}$ in Table~\ref{table:frames} (it is the vertex three edges to the right of $k.v$ in Figure~\ref{fig:sigma_fk}).
\end{remark}

\subsection{Simplicial approach}\label{sec:simp_approach}

We now give a second proof of the theorem. This time the ping-pong sets will be defined using the simplicial structure of the building rather than its metric structure.

\begin{figure}[ht]
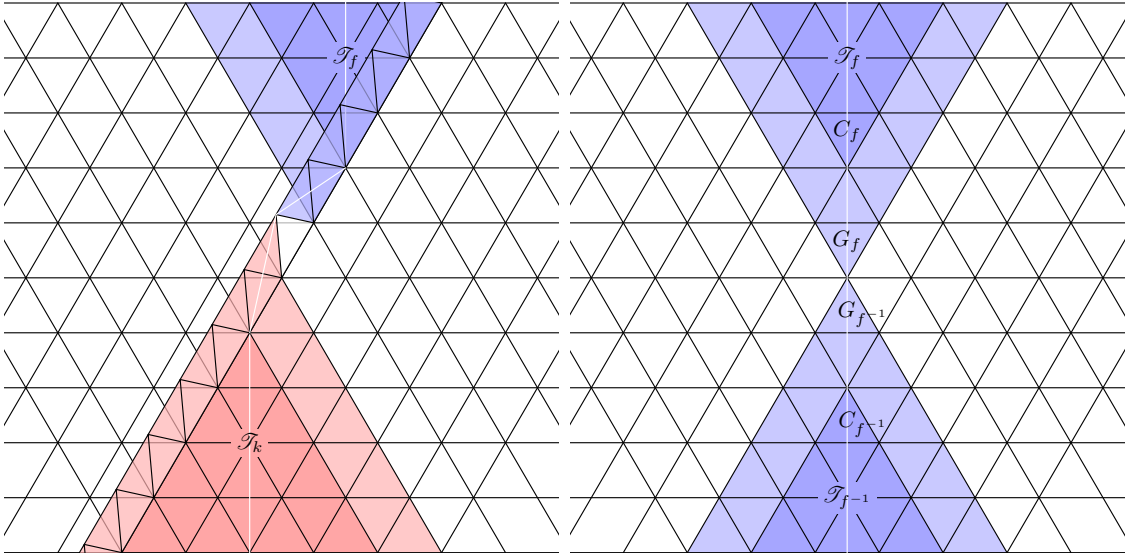

\centering
\includegraphics[width=.49\textwidth]{building_pong-0}
\includegraphics[width=.49\textwidth]{building_pong-2}
\caption{The left picture shows the sectors $\sector{T}_f$ and $\sector{T}_k$ in the extended apartment $\Sigma_{f,k}$ from Figure~\ref{fig:sigma_fk}. The larger lighter sectors are~$\sector{S}_f$ and~$\sector{S}_k$, respectively. The right picture shows $\Sigma_{f}$. The sectors~$\sector{T}_f$ and~$\sector{T}_{f^{-1}}$ are painted blue and the sectors $\sector{S}_f$ and $\sector{S}_{f^{-1}}$ are painted lighter blue. The chambers $G_{f^\pm}$ and $C_{f^\pm}$ are labeled.}
\label{fig:sectors}
\end{figure}

Let $v$ be the standard vertex in $X$. For each $g \in \{f,k,f\I,k\I\}$ we make the following definitions. We let $\sector{S}_g$ be the sector with tip $v$ that contains the attracting fixed point at infinity of $g$ in its boundary, that is, contains the points $g^n.v, n\ge 0$. Let $G_g$ be the chamber at the tip of $\sector{S}_g$. Finally, let $C_g = g.G_g$ and $\sector{T}_g = g.\sector{S}_g$; see Figure~\ref{fig:sectors}. We will make repeated use of the retraction $\rho_{\Sigma_g,C_g}$ (where $\Sigma_{g\I} = \Sigma_g$) so we abbreviate it to $\rho_g$.

We define the new ping-pong sets to be
\[
X_f \defeq \rho\I_f(\sector{T}_f) \cup \rho\I_{f\I}(\sector{T}_{f\I}) \text{ and } X_k \defeq \rho\I_{k}(\sector{T}_k) \cup \rho\I_{k\I}(\sector{T}_{k\I}) \text{.}
\]
Note that $\rho\I_g(\sector{T}_g) = g.\rho\I_{\Sigma_g,G_g}(\sector{S}_g)$ by the following observation.

\begin{observation}\label{obs:shifted_retractions}
If $g$ is a building isometry, $\Sigma$ is an apartment and $C$ is a chamber, then $g \circ \rho_{\Sigma,C} = \rho_{g.\Sigma,g.C} \circ g$.
\end{observation}

\pagebreak[3]

\begin{lemma}[Simplicial approach -- disjointness]\label{lem:retractions_disjointness}
 $X_f \cap X_k = \emptyset$.
\end{lemma}

\begin{proof}
 First note that for every $g \in \{f,k,f\I,k\I\}$ the set
 \[
 \rho\I_{\Sigma_g,G_g}(\sector{S}_g)
 \]
 consists of those chambers whose building theoretic projection onto $v$ is $G_g$. Hence their interiors are pairwise disjoint. But translating by $g$ moves each set into its interior.
\end{proof}

\begin{figure}[ht]
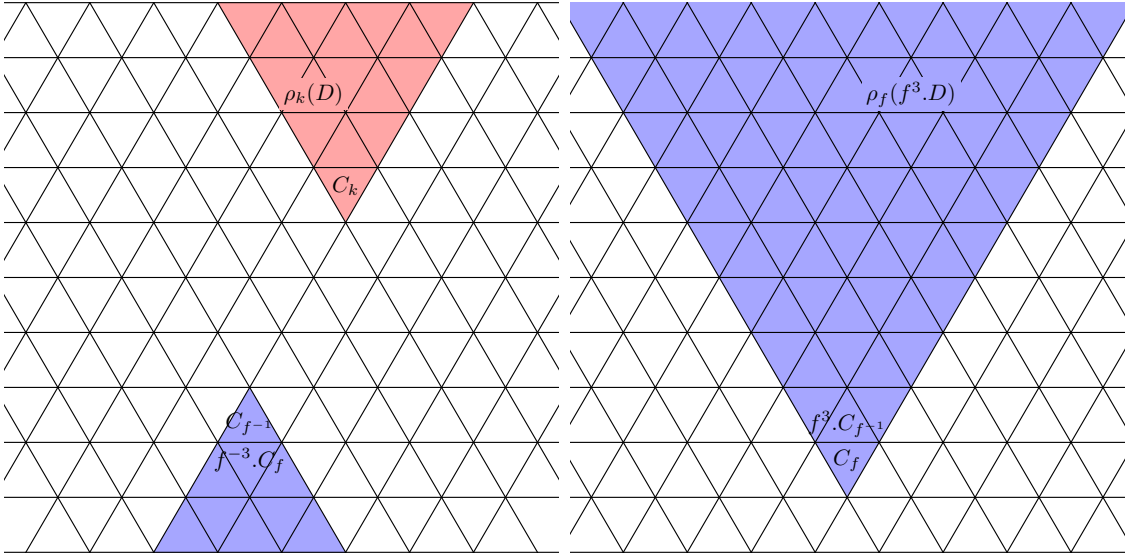

\centering
\includegraphics[width=.49\textwidth]{building_pong-3}
\includegraphics[width=.49\textwidth]{building_pong-4}
\caption{The proof of Proposition~\ref{prop:retractions_pong} for $n=3$. The left picture shows $\Sigma_{f^{-1},k}$ and the right picture shows $\Sigma_f$. There is a minimal gallery from $f^{-3}.C_f$ to $\rho_k(D)$ that passes through $C_k$, showing that $\delta(f^{-3}.C_f,\rho_k(D)) = \delta(f^{-3}.C_f,D)$. Moreover, $\delta(f^{-3}.C_f,D) = \delta(C_f,f^3.D)$ which determines $\rho_f(f^3.D)$.}
\label{fig:simplicial_ping_pong}
\end{figure}

\begin{proposition}[Simplicial approach -- ping-pong]\label{prop:retractions_pong}
 For any $|n|>2$, $f^n(X_k) \subseteq X_f$ and $k^n(X_f) \subseteq X_k$.
\end{proposition}

\begin{proof}
 We show that $f^n(\rho\I_k(\sector{T}_k)) \subseteq \rho\I_f(\sector{T}_f)$ for $n>2$, the other cases being similar. Let $D$ be a chamber such that $\rho_k(D) \in \sector{T}_k$. Since $\rho_k(D)$ lies in $\Sigma_{f\I,k}$, we can determine $\delta(f^{-3}.C_f,\rho_k(D))$ inside it. In particular, we see that a minimal gallery from $f^{-3}.C_f$ to $\rho_k(D)$ has to pass through $C_{f\I}$ and can be chosen to pass through $C_k$. The latter fact means that $\delta(f^{-3}.C_f,\rho_k(D)) = \delta(f^{-3}.C_f,D)$.

 Shifting by $f^n$, we see that every minimal gallery from $f^{n-3}.C_f$ to $f^n.D$ passes through $f^n.C_{f\I}$. As a consequence, the projection of $f^n.D$ to $f^{n-2}.v$ has to be $f^{n-2}.G_f=f^{n-3}.C_f$, and in particular (since $n>2$) the projection to $f.v$ has to be~$C_f$. Thus $f^n.D \in \rho\I_f(\sector{T}_f)$. See Figure~\ref{fig:simplicial_ping_pong} for a visual proof of the $n=3$ case.
\end{proof}

\begin{proof}[Proof 2 of Theorem]
As before, Lemma~\ref{lem:retractions_disjointness} and Proposition~\ref{prop:retractions_pong} show that $X_f$ and $X_k$ satisfy the assumptions of the Ping-Pong Lemma for $f^m$ and $k^n$ with $m,n \ge 3$.
\end{proof}

\bibliographystyle{alpha}
\bibliography{building_pong}

\end{document}